\documentclass[12pt]{amsart}
\usepackage{amsthm,amsmath,amsfonts,amssymb,latexsym,amscd,mathrsfs,hyperref,cleveref,graphicx}
\usepackage{setspace}
\usepackage{cite}
\usepackage{phonetic}
\setdisplayskipstretch{2}
\usepackage{cite}
\usepackage{crossreftools}

\makeatletter
\renewcommand{\pod}[1]{\allowbreak\mathchoice
  {\if@display \mkern 18mu\else \mkern 8mu\fi (#1)}
  {\if@display \mkern 18mu\else \mkern 8mu\fi (#1)}
  {\mkern4mu(#1)}
  {\mkern4mu(#1)}
}

\makeatletter
\@namedef{subjclassname@2020}{%
  \textup{2020} Mathematics Subject Classification}
\makeatother

\newtheorem{thm}{Theorem}
\newtheorem{lem}{Lemma}

\newtheorem*{thm*}{Theorem}
\newtheorem*{cor*}{Corollary}
\newtheorem*{prop*}{Proposition}

\theoremstyle{definition}

\newtheorem*{exa*}{Example}

\newcommand{\mc}{\mathcal}

   \def \e{\varepsilon} \def \g{\gamma}  \def \l{\lambda} \def \s{\sigma} \def \t{\theta} 

\makeatletter
\def\widebreve{\mathpalette\wide@breve}
\def\wide@breve#1#2{\sbox\z@{$#1#2$}%
     \mathop{\vbox{\m@th\ialign{##\crcr
\kern0.08em\brevefill#1{0.8\wd\z@}\crcr\noalign{\nointerlineskip}%
                    $\hss#1#2\hss$\crcr}}}\limits}
\def\brevefill#1#2{$\m@th\sbox\tw@{$#1($}%
  \hss\resizebox{#2}{\wd\tw@}{\rotatebox[origin=c]{90}{\upshape(}}\hss$}
\makeatletter

\numberwithin{equation}{section}

\renewcommand{\labelenumi}{\setlength{\labelwidth}{\leftmargin}
   \addtolength{\labelwidth}{-\labelsep}
   \hbox to \labelwidth{\theenumi.\hfill}}

\begin{document}
\title{The exceptional set for integers of the form $[p_1^c]+[p_2^c]$}
\author{Roger Baker}
\address{Department of Mathematics\newline
\indent Brigham Young University\newline
\indent Provo, UT 84602, U.S.A}
\email{baker@math.byu.edu}


 \begin{abstract}
Let $1 < c < 24/19$. We show that the number of integers $n \le N$ that cannot be written as $[p_1^c] + [p_2^c]$ ($p_1$, $p_2$ primes) is $O(N^{1-\s+\e})$. Here $\s$ is a positive function of $c$ (given explicitly) and $\e$ is an arbitrary positive number.
 \end{abstract}


\maketitle

\section{Introduction}\label{sec:intro}

Let $\mc E_c(N)$ denote the set of integers $n \le N$ which cannot be represented in the form $n = [p_1^c] + [p_2^c]$. Here $c$ is a constant, $c > 1$ and $p_1$, $p_2$ are primes. We seek to show for as large a range of $c$ as possible that the cardinality $|\mc E_c(N)|$ is of smaller order than $N$ as $N \to \infty$. There is a result of this kind due to Laporta \cite{lap}: let $c\in \left(1, \frac{17}{16}\right)$. Then for any $\e > 0$ and $B > 0$,
 \begin{equation}\label{eq:|Ec(N)|}
|\mc E_c(N)| \ll N\, \exp(-B(\log N)^{\frac 13- \e}). 
 \end{equation}
The implied constant in \eqref{eq:|Ec(N)|} depends on $B$ and $\e$.

For the shorter range $c \in \left(1, \frac{24}{23}\right)$, Zhu \cite{zhu} established a bound that saves a power of $N$; for $\e > 0$,
 \begin{equation}\label{eq:|Ec(N)|<<N}
|\mc E_c(N)| \ll N^{41/18 - 4\g/3 + \e}.
 \end{equation}
Here and below $\g = 1/c$. The implied constant depends on $c$ and $\e$ in \eqref{eq:|Ec(N)|<<N}.

In the present paper we make a power saving for a longer range of $c$.

 \begin{thm}\label{thm:Let1<c<24/19}
Let $1 < c < 24/19$ and write
 \[\s = \min\left(\frac{48-38c}{29} \, , \,
 \frac{16-10c}{75}\right).\]
We have
 \begin{equation}\label{eq:|Ec(N)|<<N1-sig+eps}
|\mc E_c(N)|\ll N^{1-\s+\e}.
 \end{equation}
The implied constant in \eqref{eq:|Ec(N)|<<N1-sig+eps} depends on $c$ and $\e$.
 \end{thm}

Our argument is an application of the Hardy-Littlewood method, with a single `major arc'
 \[\mc M := [-\omega, \omega] \, , \,
 \omega = N^{-\frac 13 - \e}.\]
The discussion of the major arc follows Zhu \cite{zhu}. On the minor arc
 \[m := [\omega, 1-\omega],\]
we need to give a good fourth power moment bound for the exponential sum
 \begin{equation}\label{eq:T(x)=sumpleNgamma}
T(x) = \sum_{p\le M^\g} \log p\, e(x[p^c]).
 \end{equation}
The calculations here are somewhat similar to material in \cite{rcb2}. A device of Cai \cite{cai} (see Lemma \ref{lem:suppan(1lenleX)} (i) below and an exponential sum bound of Heath-Brown \cite{hb2} (Lemma \ref{lem:Letkbeaninteger} below) play a key role.

 \section{Preparatory Lemmas}\label{sec:preplemmas}

We assume throughout that $1 < c < 24/19$ and that the positive number $\e$ is sufficiently small. Constants implied by `$\ll$' and `$O$' depend at most on $c$ and $\e$, unless otherwise stated. We write $A \asymp B$ when $A \ll B \ll A$.

 \begin{lem}\label{lem:Let0<x<1}
Let $0 < x < 1$. For real numbers $a_n$, $|a_n|\le 1$, let
 \[W(X,x) = \sum_{1 \le n \le X} a_n 
 e(x[n^c]).\]
Given $H$, $2 \le H \le X$, we have
 \begin{align*}
W(X,x) &\ll \frac{X\mc L}H + \sum_{0 \le h \le H} \min\left(1, \frac 1h\right) \Bigg|\sum_{1 \le n \le X} a_n e((h+\g)n^c)\Bigg|\\[2mm]
&\quad + \sum_{h=1}^\infty \min\left(\frac 1h\, , \, \frac H{h^2}\right) \Bigg|\sum_{1 \le n \le X} e(hn^c)\Bigg|.
 \end{align*}
Here $\g \in \{x, -x\}$, while (here and below), $\mc L = \log X$.
 \end{lem}

 \begin{proof}
This is a variant of \cite{rcb2}, Lemma 1, where the summation over $n$ runs over $\left(\frac X8, X\right]$. The proof is unchanged.
 \end{proof}

 \begin{lem}\label{lem:Letellge0}
Let $\ell \ge 0$ be a given integer, $L = 2^\ell$. Suppose that $f$ has $\ell + 2$ continuous derivatives on $\left[\frac X2, X\right]$ and
 \begin{equation}\label{eq:f(r)(x)|}
|f^{(r)}(x)| \asymp FX^{-r} \quad \left(r=1,\ldots, \ell + 2,\, x \in \left[\frac X2, X\right]\right).
 \end{equation}

Then for $[a,b] \subset \left[\frac X2, X\right]$, we have
 \begin{equation}\label{eq:sumalenleb e(f(n))}
\sum_{a\le n \le b} e(f(n)) \ll F^{1/(4L-2)} X^{1-(\ell+2)/(4L-2)} + F^{-1}X. 
 \end{equation}
The implied constant in \eqref{eq:sumalenleb e(f(n))} depends only on the implied constant in \eqref{eq:f(r)(x)|}.
 \end{lem}
 
 \begin{proof}
Graham and Kolesnik \cite{grakol}, Theorem 2.9.
 \end{proof}
 
 \begin{lem}[B process]\label{lem:(Bprocess)} Suppose that $f'' < 0$ and, for $x \in \left[\frac X2, X\right]$,
 \[|f''(x)|\asymp FX^{-2}, \ f^{(j)}(x) \ll FX^{-j}
 \quad (j = 3,4).\]
Define $x_\nu$ by $f'(x_\nu) = \nu$ and let $\phi(\nu) = -f(x_\nu) + \nu x_\nu$. Then for $[a,b]\subset \left[\frac X2, X\right]$, we have
 \begin{align*}
\sum_{a \le n \le b} e(f(n)) &= \sum_{f'(b) \le \nu \le f'(a)}  \frac{e(\phi(\nu)-1/8)}{|f''(x_\nu)|^{1/2}}\\[2mm]
&\hskip .5in + O(\log(FX^{-1}+2) + F^{-1/2}X).
 \end{align*}

 \end{lem}
 
 \begin{proof}
\cite{grakol}, Lemma 3.6.
 \end{proof}

 \begin{lem}\label{lem:Letkbeaninteger}
Let $k$ be an integer, $k \ge 3$. Let $f$ have continuous derivatives $f^{(j)}$ $(1 \le j \le k)$ on $[0,N]$,
 \[|f^{(k)}(x)| \asymp \l_k \quad \text{ on } (0, N].\]
Then for $(a,b] \subset (0,N]$,
 \[\sum_{a \le n \le b} e(f(n)) \ll N^{1+\e}
 \left(\l_k^{\frac 1{k(k-1)}} + N^{-\frac 1{k(k-1)}}
 + N^{-\frac 2{k(k-1)}} \l_k^{-\frac 2{k^2(k-1)}}\right).
 \]
 \end{lem}
 
 \begin{proof}
Heath-Brown \cite{hb2}, Theorem 1.
 \end{proof}
 
 \begin{lem}\label{lem:Let0<B<1}
Let $0 < B < 1$ and $|c_n| \le 1$. Let
 \[V(x) = \sum_{1 \le n \le X} c_n e([n^c]x).\]
Then
 \[\int_B^{2B} |V(y)|^2 dy \ll XB + X^{2-c}\mc L.\]
 \end{lem} 

 \begin{proof}
\cite{rcb2}, Lemma 12.
 \end{proof}

Let $\|\t\|$ denote distance of $\t$ from the nearest integer.

 \begin{lem}\label{lem:suppan(1lenleX)}
(i) Suppose that $a_n$ $(1 \le n \le X)$ are complex numbers with $a_n \ll \mc L$. Let
 \[V(n) = \sum_{1 \le n \le X} a_n e(x[n^c]).\]
Suppose that $U > 0$ satisfies
 \[\sum_{1\le n \le X} e(x[n^c]) \ll U + \mc LX^{1-c}
 \|x\|^{-1} \quad (0 < |x| \le 2).\]
Then for any Borel measurable bounded function $G$ on $[\omega, 1-\omega]$, we have
 \begin{align}
&\left|\int_\omega^{1-\omega} V(x)G(x) dx\right|^2 \label{eq:|intomega1-omegaV(x)G(x)dx|2}\\[2mm]
&\qquad \ll \mc L^4 X^{2-c} \int_\omega^{1-\omega} |G(x)|^2 dx + \mc L^2UX \left(\int_\omega^{1-\omega} |G(x)| dx\right)^2.\notag
 \end{align}
(ii) Suppose further that
 \[V(x) \ll V \quad (x\in [\omega, 1-\omega]).\]
Then
 \[\int_\omega^{1-\omega} |V(x)|^4 dx \ll
 \mc L^4(V^2 X^{2-c} + UX^2).\]
 \end{lem}
 
 \begin{proof}
Arguing as in \cite{rcb2}, proof of Lemma 13, the left-hand side of \eqref{eq:|intomega1-omegaV(x)G(x)dx|2} is
 \begin{align*}
&\ll X\mc L^2 \int_\omega^{1-\omega} |G(y)| \int_\omega^{1-\omega} |G(x)| U\, dx\\[2mm]
&+ X\mc L^3 \int_\omega^{1-\omega} |G(y)| \int_\omega^{1-\omega} |G(x)| \min \left(X, \frac{X^{1-c}}{\|x-y\|}\right) dxdy. 
 \end{align*}
It now suffices to show that
 \begin{align}
\int_\omega^{1-\omega} &|G(y)| \int_\omega^{1-\omega} |G(x)| \min \left(X, \frac{X^{1-c}}{\|x-y\|}\right) dxdy \label{eq:intomega1-omega|G(y)|}\\[2mm]
&\ll X^{1-c} \mc L \int_\omega^{1-\omega} |G(x)|^2 dx.\notag
 \end{align}
The left-hand side of \eqref{eq:intomega1-omega|G(y)|} is
 \begin{align*}
&\le \frac 12 \int_\omega^{1-\omega} \int_\omega^{1-\omega} (|G(x)|^2 + |G(y)|^2) \min \left(X, \frac{X^{1-c}}{\|x-y\|}\right)dxdy\\[2mm]
&= \int_\omega^{1-\omega} |G(x)|^2 \int_\omega^{1-\omega} \min\left(X, \frac{X^{1-c}}{\|x-y\|}\right) dxdy.
 \end{align*}
It is a straightforward matter to show that the last inner integral is $\ll X^{1-c}\mc L$, and \eqref{eq:intomega1-omega|G(y)|} follows.

For part (ii) we take $G(x) = \overline V(x)|V(x)|^2$. Now
 \[\int_\omega^{1-\omega}|G(x)|^2dx \ll V^2 
 \int_\omega^{1-\omega} |V(x)|^4dx\]
while Cauchy's inequality together with Lemma \ref{lem:Let0<B<1} yields
 \begin{align*}
\int_\omega^{1-\omega} |G(x)|dx &\le \int_\omega^{1-\omega} |V(x)|^2dx \int_\omega^{1-\omega} |V(x)|^4 dx\\[2mm]
&\ll X\mc L^2 \int_\omega^{1-\omega} |V(x)|^4dx.
 \end{align*}
Now part (i) yields
 \[\left(\int_\omega^{1-\omega} |V(x)|^4 dx\right)^2
 \ll \int_\omega^{1-\omega} |V(x)|^4 dx \ \mc L^4
 (X^{2-c} V^2 + UX^2),\]
and (ii) follows.
 \end{proof}
 
 \begin{lem}\label{lem:LetGbeacomplexfunction}
Let $G$ be a complex function on $[1,X]$. Let $u \ge 1$, $v$, $z$ be numbers satisfying $u^2 \le z$, $128 uz^2 \le X$, and $2^{20}X \le v^3$. Then
 \[\sum_{1 \le n \le X} \Lambda(n) G(n)\]
is a linear combination (with bounded coefficients) of $O(\mc L)$ sums of the form

 \[\sum_m a_m \sum_{n\ge z} (\log n)^4 G(mn)\]
with $h=0$ or 1, $|a_m| \le d(m)^5$, together with $O(\mc L^3)$ sums of the form
  \[\underset{1 \le mn \le X}{\sum_m a_n \sum_{u\le n \le v}}
  b_n G(mn)\]
in which $|a_m| \le d(m)^5$, $|b_n| \le d(n)^5$. Here $d(\cdots)$ is the divisor function.
 \end{lem}
 
 \begin{proof}
Heath-Brown \cite{hb1}, pp.~1367--1368.
 \end{proof}

Throughout the rest of the paper, let $2 \le M \le N$. We write $X = M^\g$ and recall the definition of $T(x)$ in \eqref{eq:T(x)=sumpleNgamma}.
 
 \begin{lem}\label{lem:Wehave,fornin[M/2,M]}
We have, for $n \in [M/2, M]$,
 \[\int_{-\omega}^\omega T(x)^2 e(-xn)dx =
 \frac{\Gamma^2(1+\g)}{\Gamma(2\g)}\, n^{2\g-1}
 + 0 \left(M^{2\g-1} \exp\left(-a\left(
 \frac{\log M}{\log\log M}\right)^{1/3}\right)\right),\]
where $a$ is a positive constant. The implied constant depends on $c$ and $a$.
 \end{lem}
  
 \begin{proof}
This is Lemma 3.6 of Zhu \cite{zhu}, with minor corrections provided on p.~197 of the erratum.
 \end{proof}
 \bigskip
 
 \section{Proof of Theorem \ref{thm:Let1<c<24/19}}

Let $R(n)$ be the number of representations $n = [p_1^c] + [p_2^c]$, $p_1$ and $p_2$ prime. We proceed initially as on p.~198 of Zhu \cite{zhu}. For our theorem, it suffices by a dyadic argument to show that when $M \le N$, the quantity
 \[Z(M) := \left|\left\{n \in \left(\frac M2, M\right]
 : \ R(n) = 0\right\}\right|.\]
satisfies
 \[Z(M) \ll M^{1-\s + \e}.\]
Let $n \in \left(\frac M2, M\right]$. If $R(n) = 0$, then
 \[\int_{-\omega}^{1-\omega} T(x)^2 e(-xn) dx = 0,\]
and Lemma \ref{lem:Wehave,fornin[M/2,M]} yields
 \[\left|\int_{m} T(x) e(-xn)dx\right|= \left|
 \int_{\mc M} T(x)^2 e(-xn)dx\right|
 \gg M^{2\g -1}.\]
Thus
 \begin{align*}
Z(M)M^{4\g-2} &\ll \sum_{n\in Z(M)} \left|\int_{m} T(x)^2 e(-xn)dx\right|^2\\[2mm]
&\ll \int_{m} |T(x)|^4 dx,
 \end{align*}
where Bessel's inequality is used in the last step. It remains to show that
 \[M^{2-4\g} \int_m |T(x)^4dx \ \ll \ M^{1-\s+\e},\]
so that it suffices to show that, with $X = M^\g$,
 \begin{equation}\label{eq:intm|T(x)|4dx}
\int_m |T(x)|^4 dx \ll X^{4-c-c\s + \e}. 
 \end{equation}
 
Let us assume for a moment that
 \begin{equation}\label{eq:sum1lenleXe(x[nc])}
\sum_{1\le n \le X} e(x[n^c]) \ll X^{2-c-c\s + \e/2} + X^{1-c} \|x\|^{-1} \quad (0 < x \le 2)
 \end{equation}
and that
 \begin{equation}\label{eq:T(x)llX1-cs/2+e/4}
T(x) \ll X^{1-c\s/2+\e/4} \quad (\omega \le x \le 1-\omega).
 \end{equation}
Lemma \ref{lem:suppan(1lenleX)} (ii) yields
 \[\int_m |T(x)|^4dx \ll \mc L^4 (X^{2-c\s + \e/2+2-c}
 + X^{2-c-c\s + \e/2 + 2}),\]
so that \eqref{eq:intm|T(x)|4dx} holds. It remains to prove \eqref{eq:sum1lenleXe(x[nc])}, \eqref{eq:T(x)llX1-cs/2+e/4}. Using a dyadic argument, we need only show (with general $X \ge 2$) that
 \begin{equation}\label{eq:sumfrace(x[nc])}
\sum_{\frac X2 < n \le X} e(x[n^c]) \ll X^{2-c-c\s + \e/2} + X^{1-c} \|x\|^{-1} \quad (0 < x < 2)
 \end{equation}
and that
 \begin{equation}\label{eq:sumfrac(logp)}
\sum_{\frac X2 < n \le X} (\log p) e(x[p^\s]) \ll X^{1-c\s/2 + \e/4} \quad (\omega \le x \le 1-\omega).
 \end{equation}
Since $c\s < 1$, we may show in place of \eqref{eq:sumfrac(logp)} that
 \begin{equation}\label{eq:sumfracLambda(n)}
\sum_{\frac X2 < n \le X} \Lambda(n) e(x[n^\s]) \ll X^{1-c \s/2 + \e/4}.
 \end{equation}

For \eqref{eq:sumfrace(x[nc])}, we apply Lemma \ref{lem:Let0<x<1}, taking $H= X^{c-1+c\s}$. Let
 \[S_h = \sum_{\frac X2 < n \le X} e((h+\g)n^c) \ ,
 \ h \ge 0,\]
where $\g = \{x\}$ if $h=0$ and $\g \in \{0, \{x\}, -\{x\}\}$ for $h \ge 1$. It suffices to show that
 \begin{align}
\sum_{h=0}^\infty \min\left(\frac 1{h+1} \, , \, \frac H{h^2}\right) |S_h| &\ll X^{2-c-c\s + \e/2} + X^{1-c} \|x\|^{-1}\label{eq:sumh=0inftymin}\\
&\hskip .5in (0 < x < 2).\notag
 \end{align}

The term $X^{1-c}\|x\|^{-1}$ on the right-hand side of \eqref{eq:sumh=0inftymin} is needed only when $h=0$ or 1 and $X^{c-1}(h+\g)<\e$. In this case, the Kusmin-Landau theorem (\kern-4pt\cite[Theorem 2.1]{grakol}) gives
 \[S_h \ll X^{1-c} (h+\g)^{-1} \ll X^{1-c}\|x\|^{-1}.\]
Let
 \[F_h = (h+\g) X^c \ , \ N_h = F_hX^{-1}.\]
To bound $S_h$ we apply the B-process followed by Lemma \ref{lem:Letkbeaninteger} with $k=5$. (It is crucial here that $\frac c{c-1} > 4$ for the fifth derivative to have appropriate order of magnitude.) The error term arising from the B-process is $\ll X^{1/2}$ since $F_hX^{-1} \gg 1$. This gives rise to an acceptable contribution to \eqref{eq:sumh=0inftymin} since $2 - c - c\s > 1/2$. In order to obtain \eqref{eq:sumh=0inftymin} it remains to show that
 \begin{align*}
\sum_{h=0}^\infty \min\left(\frac 1{h+1}\, , \, \frac H{h^2}\right) X^{1+\e/2} F_h^{-1/2} & N_h((F_hN_h^{-5})^{\frac 1{20}} + N_h^{-1/20}+(F_hN_h^{-5})^{-1/50}N_h^{-1/10})\\
&\qquad \ll X^{2-c-c\s + \e/2}.
 \end{align*}
Now

 \begin{align*}
X^{1+\e/2} F_h^{-1/2} N_h(F_hN_h^{-5})^{1/20} &\ll X^{1/4+\e/2} F_h^{3/10}\\[2mm]
&\ll X^{1/4 + 3c/10 + \e/2} (h+1)^{3/10},\\[2mm]
\sum_{h=0}^\infty  \min\left(\frac 1{h+1}\, , \, \frac H{h^2}\right)& X^{1/4+3c/10+\e/2}(h+1)^{3/10}\\[2mm]
&\ll H^{3/10} X^{1/4 + 3c/10 + \e/2}\\[2mm]
&\ll X^{\frac 3{10}\, (c-1+c\s)+1/4+3c/10+\e/2}\\[2mm]
&\ll X^{2-c-c\s + \e/2}
 \end{align*}
since $\s \le (41 - 32c)/26c$ is easily deduced from the definition of $\s$.

Next,
 \begin{align*}
X^{1+\e/2}F_h^{-1/2} N_h^{19/20} &\ll F_h^{9/20} X^{1/20 + \e/2}\\[2mm]
&\ll (h+1)^{9/20} X^{(9c+1)/20+\e/2},\\[2mm]
\sum_{h=0}^\infty \min\left(\frac 1{h+1}\, , \, \frac H{h^2}\right) & (h+1)^{9/20} X^{(9c+1)/20+\e/2}\\[2mm]
&\ll X^{\frac 9{20}\, (c-1+c\s)+(9c+1)/20+\e/2}\\[2mm]
&\ll X^{2-c-c\s + \e/2}
 \end{align*}
since
 \[\s \le \frac{48-38c}{29c}\, .\]

Finally,
 \begin{align*}
X^{\e/2} F_h^{-1/2} (F_h N_h^{-5})^{-1/50}& N_h^{9/10} = X^{\e/2} F_h^{24/50}\\[2mm]
&\ll (h+1)^{12/25} X^{12c/25 + \e/2},\\[2mm]
\sum_{h=0}^\infty \min\left(\frac 1{h+1}\, , \, \frac H{h^2}\right) & (h+1)^{12/25} X^{12c/25+\e/2}\\[2mm]
&\ll X^{(c-1+c\s)12/25} X^{12c/25 + \e/2}.
 \end{align*}
This is
 \[\ll X^{2-c-c\s+\e/2}\]
since $\s \le (62-49c)/37c$ follows from the definition of $\s$. Thus \eqref{eq:sumh=0inftymin} holds.

In order to prove \eqref{eq:sumfracLambda(n)}, we apply Lemma \ref{lem:LetGbeacomplexfunction}, taking $u = \frac{X^{1/5}}{128}$, $z = X^{2/5}$, $v = 128X^{1/3}$. Recall that $\omega \le x \le 1 - \omega$. It suffices to show that, for $|a_m| \le 1$,
 \[\underset{mn\le X}{\sum_m a_m \, \sum_{n \ge X^{2/5}}} 
 (\log n)^h e(x[(mn)^c]) \ll X^{1-\frac{c\s}2 + \frac \e 5}\]
\textit{and} that, for $|a_m| \le 1$, $|b_n| \le 1$ we have
 \[\underset{mn\le X}{\sum_m a_m \sum_{X^{1/5}
 \le n \le 2^7 x^{1/3}}} b_n e(x([(mn)^c])) \ll
 X^{1-c\s/2 + \e/5}.\]
After applying Lemma \ref{lem:Let0<x<1} with $H = X^{c\s/2}$, a partial summation if needed, and dyadic dissections we see that it suffices to show
 \begin{equation}\label{eq:sumh=0inftymin(1(h+1)}
\sum_{h=0}^\infty \min\left(\frac 1{h+1}\, , \, \frac H{h^2}\right) |S_j(h)| \ll X^{1-\frac{c\s}2 + \e/6} \quad (j=1,2).
 \end{equation}

Here
 \[S_1(h) = \sum_{m\le XY^{-1}} \Bigg|
 \sum_{\substack{Y \le n < Y'\\
 mn \le X}} e((h+\g)(mn)^c\Bigg|\]
where $Y\ge X^{2/5}$, $Y< Y' \le 2Y$ and
 \begin{equation}\label{eq:gamma=x if h=0}
\g = x \, \text{ if }\, h = 0 \ , \ \g \in \{x,-x\} \, \text{ if } \, h \ge 1.
 \end{equation}
As for $S_2(h)$,
 \[S_2(h) = \sum_{Z \le m \le 2Z} \Bigg|
 \sum_{\substack{Y < n \le 2Y\\
 mn \le X}} b_n e((h + \g)(mn)^c)\Bigg|\]
with $Z \le X$ and $X^{1/5} \ll Y \ll X^{1/3}$, $ZY \ll X$.

To bound $S_1(h)$ we apply Lemma \ref{lem:Letellge0} with $\ell = 2$ to the inner sum. Let $F = (h+\g)X^c$, then
 \begin{align*}
F \gg xX^c &\gg X^{c-1/3-\e},\\[2mm]
XY^{-1}(YF^{-1}) &= XF^{-1} \ll X^{4/3-c+\e} \ll X^{1-c\s}
 \end{align*}
since $c\s < 6/75$ from the definition of $\s$. This term produces a satisfactory contribution in \eqref{eq:sumh=0inftymin(1(h+1)}. Also
 \begin{align*}
F \ll (h+1)& X^c,\\[2mm]
XY^{-1} F^{1/14} Y^{5/7} &\ll (h+1)^{1/14} X^{1+c/14 - \frac 27 \cdot \frac 25},\\[2mm]
\sum_{h=0}^\infty \min\Bigg(\frac 1{h+1}\, &, \, \frac H{h^2}\Bigg) |S_1(h)|\\[2mm]
&\ll H^{1/14} X^{31/35 + c/14}\\[2mm]
&\ll X^{\frac{31}{35} + \frac c{14} + \frac{c\s}{28}} \ll X^{1-c\s/2}
 \end{align*}
since
 \[\s \le \frac{16-10c}{75c}\, .\]

For $S_2(h)$, we proceed as in the proof of Theorem 5 of \cite{rcb1}. We have, for $Q \le N$,
 \begin{equation}\label{eq:|S2(h)|2llfracX2Q}
|S_2(h)|^2 \ll \frac{X^2}Q + \frac XQ \ \sum_{q \le Q} \ \sum_{n \asymp Y} \Bigg|\sum_{m\in I(n)} e((h+\g)m^c((n+q)^c - n^c))\Bigg|
 \end{equation}
where $I(n)$ is a subinterval of $[Z, 2Z]$. Take $Q = X^{c\s} \le X^{6/75}$.

We estimate the inner sum in \eqref{eq:|S2(h)|2llfracX2Q} using Lemma \ref{lem:Letellge0} with $\ell = 1$. We have
 \begin{align*}
\sum_{m\in I(n)} &e((h+\g)m^c((n+q)^c-n^c))\\[2mm]
&\quad \ll ((h+1)X^{c\s + c}Y^{-1})^{1/6} Z^{1/2} + ZF_{h,q}^{-1}\\
\intertext{where}
F_{h,q} &= (h+\g) X^c\, q\, Y^{-1} \ge X^{c-2/3-\e} \ge X^{c\s}
 \end{align*}
leading to an acceptable term arising from $ZF_{h,q}^{-1}$ in \eqref{eq:sumh=0inftymin(1(h+1)}. (Of course, $Q$ was chosen so that $X^2/Q$ is likewise acceptable.)

Now a short computation yields
 \begin{align*}
\frac XQ\ \sum_{q\le Q} \ \sum_{n\le Y} ((h+1) &X^{c\s + c} Y^{-1})^{1/6} Z^{1/2}\\
& \ll  (h+1)^{1/6} X^{29/18+(c\s+c)/6}
 \end{align*}
(since $Y \ll X^{1/3}$). The contribution to \eqref{eq:sumh=0inftymin(1(h+1)} is
 \begin{align*}
\ll \sum_{h=0}^\infty &\min \left(\frac 1{h+1}\, , \, \frac H{h^2}\right) (h+1)^{1/12} X^{29/36+(c\s + c)/12}\\[2mm]
&\hskip .5in \ll X^{29/36 + (3c\s + 2c)/24} \ll X^{1-c\s/2}
 \end{align*}
since $\s \le \frac{14-6c}{45c}$ from the definition of $\s$. This shows that \eqref{eq:sumh=0inftymin(1(h+1)} holds, and the proof of Theorem \ref{thm:Let1<c<24/19} is complete. $\hfill\Box$

 \end{document}